\documentclass[12pt]{amsart}
\usepackage{amssymb}
\usepackage{amsfonts}
\usepackage{latexsym}
\usepackage{amscd}
\usepackage[mathscr]{euscript}
\usepackage{xy} \xyoption{all}

\newcommand{\N}{{\mathbb{N}}}

\newcommand{\uloopr}[1]{\ar@'{@+{[0,0]+(-4,5)}@+{[0,0]+(0,10)}@+{[0,0] +(4,5)}}^{#1}}
\newcommand{\uloopd}[1]{\ar@'{@+{[0,0]+(5,4)}@+{[0,0]+(10,0)}@+{[0,0]+ (5,-4)}}^{#1}}
\newcommand{\dloopr}[1]{\ar@'{@+{[0,0]+(-4,-5)}@+{[0,0]+(0,-10)}@+{[0, 0]+(4,-5)}}_{#1}}
\newcommand{\dloopd}[1]{\ar@'{@+{[0,0]+(-5,4)}@+{[0,0]+(-10,0)}@+{[0,0 ]+(-5,-4)}}_{#1}}

\newcommand{\luloop}[1]{\ar@'{@+{[0,0]+(-8,2)}@+{[0,0]+(-10,10)}@+{[0, 0]+(2,2)}}^{#1}}

\newtheorem{lemma}{Lemma}
\newtheorem{corollary}[lemma]{Corollary}
\newtheorem{theorem}[lemma]{Theorem}
\newtheorem{proposition}[lemma]{Proposition}
\newtheorem{remark}[lemma]{Remark}
\newtheorem{definition}[lemma]{Definition}

\newtheorem*{nota}{Notation}

\def\Bez{B\'{e}zout}

\def\N{{\mathbb N}}

\def\dualita#1#2{\mathrel{
                 \mathop{\vcenter{
                 \offinterlineskip
                 \hbox to 0.6truecm{\rightarrowfill}
                 \hbox to 0.6truecm{\leftarrowfill}}}%
                 \limits_{#2}^{#1}}}

\begin{document}

\title[Leavitt path algebras are B\'{E}ZOUT]{Leavitt path algebras are B\'{e}zout}
\author{Gene Abrams$^*$}
\address{Department of Mathematics, University of Colorado,
Colorado Springs, CO 80918 U.S.A.}
\email{abrams@math.uccs.edu}
\thanks{
.  \ \  $^*$corresponding author \ \ abrams@math.uccs.edu  \\  The first  author is partially supported by a Simons Foundation Collaboration Grants for Mathematicians Award \#208941.     The second and third authors are supported by Progetto di Eccellenza Fondazione Cariparo ``Algebraic structures
and their applications: Abelian and derived categories, algebraic entropy and representation of algebras''.
}

\author{Francesca Mantese}
\address{Dipartimento di Informatica, Universit\`{a} degli Studi di Verona, I-37134 Verona, Italy}
\email{francesca.mantese@univr.it}

\author{Alberto Tonolo}
\address{Dipartimento Matematica, Universit\`{a} degli Studi di Padova, I-35121, Padova, Italy}
\email{tonolo@math.unipd.it}


\dedicatory{Dedicated to the memory of Frank W. Anderson: teacher, scholar, friend}


\begin{abstract}
Let $E$ be a  directed graph, $K$ any field, and let $L_K(E)$ denote  the Leavitt path algebra of $E$ with coefficients in $K$.     We show that $L_K(E)$ is a \Bez \ ring, i.e.,  that every finitely generated one-sided ideal of $L_K(E)$ is principal.  
\end{abstract}

\maketitle

Given any directed graph $E$ and field $K$, one may construct the {\it Leavitt path algebra of E with coefficients in K} (denoted $L_K(E)$), as first described in \cite{AAP1} and \cite{AMP}.  Since the introduction of the topic in 2005, Leavitt path algebras have been the object  of broad and deep investigation,   not only among algebraists, but in the analysis and symbolic dynamics communities as well.   The synergy within  this at-first-seemingly-disparate group of researchers has led to a number of important advances in each of these disciplines.   For a general overview of the topic, including a description of these wide-ranging connections, see e.g. \cite{Survey}.

Over the past decade, various structural properties of the algebras $L_K(E)$ have been discovered, with many of the results in the subject taking on the following form:  $L_K(E)$ has some specified algebraic property if and only if $E$ has some specified graph-theoretic property.  (The structure of the field $K$ often plays no role in results of this type.)  A few (of many) examples of such results include a description of those Leavitt path algebras which are simple; purely infinite simple; finite dimensional; prime; primitive;  etc.        A significant majority of these results can be interpreted as statements about the two-sided ideal structure of $L_K(E)$.   Indeed,   it has only been quite recently  that the subject of the one-sided structure (and, more generally,  the module-theoretic structure) of these algebras has been broached (see e.g. \cite{AR}).    Herein we continue this  line of investigation by establishing a perhaps-surprising result about the structure of the finitely generated one-sided ideals of a Leavitt path algebra, as follows.

An ideal $I$ of a ring $R$ is called {\it principal} if it is generated by a single element, i.e., if $I = Rt$ for some $t\in I$.   The well-studied notion of a {\it principal ideal domain} (a commutative integral domain in which every ideal is principal) appears throughout the mathematical literature; the standard nontrivial example of such is the algebra of polynomials $K[x]$ over a field $K$.    More generally, a {\it  principal ideal ring} is a (not necessarily commutative, not necessarily zero-divisor-free) ring in which every one-sided ideal is principal.  These too have been well-studied, with an overview appearing in  \cite[Chapter 3]{J}, and foundational results appearing in,  e.g.,  \cite{Amitsur},   \cite{Asano},  and \cite{Goldie}.  

 In particular, a principal ideal ring necessarily contains no non-finitely-generated one-sided ideals.    So the notion of a principal ideal ring can be extended to include only those one-sided ideals which admit the possibility  of being principal:

\medskip

{\bf Definition.}  A ring $R$ is called \emph{left B\'{e}zout} in case every finitely generated left ideal of $R$ is principal.  

\medskip

$R$ is called \emph{\Bez}  in case every finitely generated one-sided (left and right) ideal of $R$ is principal.  
 As with their more classical  forebears, \Bez \ rings have been investigated in a number of settings.       For instance,    Kaplansky considered \Bez \ rings in \cite[\S 4]{Kap},  as part of his investigation of triangular reduction of matrices.   Various additional properties of  \Bez \  rings  have been  considered as well in \cite{B}, \cite{Cohn},    \cite{Rob},  and  \cite{W}.

There are a number of ring-theoretic properties which $R = L_K(E)$ possesses for any  graph $E$.  These include:  $R$ is hereditary (every one-sided ideal is projective), $R$ is  semiprimitive (zero Jacobson radical), and $R$ is ring-isomorphic to its opposite ring $R^{op}$. 


The main result of this article is  Theorem \ref{LpasareBezout} (and its generalization, Corollary \ref{LpasBezoutallgraphs}), in which we add another property to this  list.  

\medskip

{\bf Theorem.}   For any  graph $E$ and field $K$, the Leavitt path algebra $L_K(E)$ is \Bez. 

\medskip

We note that a key role in establishing this theorem  will be played by a construction motivated directly by a result coming from  symbolic dynamics (Theorem \ref{outsplittheorem}).   

  Until the closing remarks of this article, we will focus solely on establishing this result for finite graphs.  That the result  holds for all graphs 
 will then follow as an easy consequence of a general result about Leavitt path algebras.

 
 \smallskip

 We set some notation.  A (directed) graph $E = (E^0, E^1, s,r)$ consists of a {\it vertex set} $E^0$, an {\it edge set} $E^1$, and {\it source} and {\it range} functions $s, r: E^1 \rightarrow E^0$.  For $v\in E^0$, the set of edges $\{ e\in E^1 \ | \ s(e)=v\}$ is denoted by $s^{-1}(v)$, and the set of edges $\{ e\in E^1 \ | \ r(e)=v\}$ by $r^{-1}(v)$. $E$ is called {\it finite} in case both $E^0$ and $E^1$ are finite sets.  A vertex $v$ is a {\it source vertex} (or simply a {\it source}) in case $r^{-1}(v) = \emptyset$, i.e., in case $v$ is not the range vertex of any edge in $E$.  
A {\it path} $\alpha$ in $E$ is a sequence $e_1 e_2 \cdots e_n$ of edges in $E$ for which $r(e_i) = s(e_{i+1})$ for all $1 \leq i \leq n-1$.  
We say that such $\alpha$ has {\it length} $n$, and we write $s(\alpha) = s(e_1)$ and $r(\alpha) = r(e_n)$.  We view each vertex $v \in E^0$ as a path of length $0$, and denote $v = s(v) = r(v)$.  
A  path $c = e_1 e_2 \cdots e_n$ in $E$   is a {\it cycle}  in case $r(e_n) = s(e_1)$, and there are no repeated vertices in the set $c^0 := \{s(e_1), s(e_2), \dots, s(e_n) \}$.     A subset $W$ of $E^0$ is called {\it hereditary} in case, whenever $w\in W$ and $v\in E^0$ and there is a path $\sigma$ for which $s(\sigma) = w$ and $r(\sigma) = v$, then $v\in W$.
 We say that $W$ is {\it saturated} if whenever $s^{-1}(v)\neq \emptyset$ and $\{r(e) \ |  s(e) = v\}\subseteq W$, then $v\in W$.

\smallskip


We give here a basic description of $L_K(E)$; for additional information, see 
\cite{TheBook}.       Let $K$ be a field, and let $E = (E^0, E^1, s,r)$ be a directed  graph with vertex set $E^0$ and edge set $E^1$.   The {\em Leavitt path $K$-algebra} $L_K(E)$ {\em of $E$ with coefficients in $K$} is  the $K$-algebra generated by a set $\{v\mid v\in E^0\}$, together with a set of symbols $\{e,e^*\mid e\in E^1\}$, which satisfy the following relations:

\smallskip

(V)   \ \ \  \ $vu = \delta_{v,u}v$ for all $v,u\in E^0$, \  

  (E1) \ \ \ $s(e)e=er(e)=e$ for all $e\in E^1$,

(E2) \ \ \ $r(e)e^*=e^*s(e)=e^*$ for all $e\in E^1$,

 (CK1) \ $e^*e'=\delta _{e,e'}r(e)$ for all $e,e'\in E^1$, and

(CK2)Ê\ \ $v=\sum _{\{ e\in E^1\mid s(e)=v \}}ee^*$ for every   $v\in E^0$ for which $0 < |s^{-1}(v)| < \infty$. 

\noindent 
An alternate description of $L_K(E)$ may be given as follows.  For any graph $E$ let $\widehat{E}$ denote the ``double graph" of $E$, gotten by adding to $E$ an edge $e^*$ in a reversed direction for each edge $e\in E^1$.   Then $L_K(E)$ is the usual path $K$-algebra $K\widehat{E}$, modulo the ideal generated by the relations (CK1) and (CK2).

\smallskip

It is easy to show that $L_K(E)$ is unital if and only if $|E^0|$ is finite; in this case, $1_{L_K(E)} = \sum_{v\in E^0}v$.    Every element of $L_K(E)$ may be written as $\sum_{i=1}^n k_i \alpha_i \beta_i^*$, where $k_i$ is a nonzero element of $K$, and each of the $\alpha_i$ and $\beta_i$ are paths in $E$ with $r(\alpha_i) = r(\beta_i)$.  If $\alpha \in {\rm Path}(E)$ then we may view $\alpha \in L_K(E)$, and will often refer to such $\alpha$ as a {\it real path} in $L_K(E)$; analogously, for $\beta = e_1 e_2 \cdots e_n \in {\rm Path}(E)$ we often refer to the element $\beta^* = e_n^* \cdots e_2^* e_1^*$ of $L_K(E)$ as a {\it ghost path} in $L_K(E)$.     The  map $KE \rightarrow L_K(E)$ given by the $K$-linear extension of $\alpha \mapsto \alpha$ (for $\alpha \in {\rm Path}(E)$) 
 is an injection of $K$-algebras by \cite[Corollary 1.5.12]{TheBook}.

 Since for a Leavitt path algebra we have $L_K(E) \cong L_K(E)^{op}$, the properties {\it left} \Bez \ and {\it right}  \Bez \ are equivalent in this context, so we will simply say that $L_K(E)$ is \Bez \ in either case.   We will use the language of left modules throughout.

\begin{remark}\label{Bezoutremark}
{\rm
In making conjectures about structural properties of Leavitt path algebras, one always first  tests those conjectures against the three primary colors of algebras which arise as such:  the $n \times n$ matrix rings  ${\rm M}_n(K)$; the Laurent polynomial ring $K[x,x^{-1}]$; and the classical Leavitt algebras  $L_K(1,n)$.   A matrix ring ${\rm M}_n(K)$ is well-known to be B\'{e}zout (in fact, to be a principal ideal ring);   this is true as well for the principal ideal domain $K[x,x^{-1}]$ and for the principal ideal ring ${\rm M}_n(K[x,x^{-1}])$ (for a proof that  ${\rm M}_n(K[x,x^{-1}])$ is a principal ideal ring see for instance \cite[Theorem 40]{J} ).   Intuitively, these are ``well-behaved" examples.   The somewhat exotic nature of  $L_K(1,n)$ provides the reason that it too is  \Bez.  Specifically, let $I$ be a finitely generated left ideal of  $R=L_K(1,n)$.  Since any Leavitt path algebra is hereditary, $I$ is finitely generated projective.  So $I$ is isomorphic to a direct summand of $R^m$ for some $m\geq 1$.   But $R$ has $R\cong R^m$ as left $R$-modules for any $m\geq 1$, so that $I$ is a direct summand of $R$, and thus principal.   Note that unlike the well-behaved cases ${\rm M}_n(K)$ and  $K[x,x^{-1}]$, $R = L_K(1,n)$ in fact has the stronger property that every finitely generated  left $R$-{\it module} is principal.  On the other hand, $L_K(1,n)$ is not a principal ideal ring; for instance, it is not hard to construct left ideals of $L_K(1,n)$ which are direct sums of infinitely many nonzero ideals.    (See Proposition \ref{pirLpas} below.)  

A fourth primary color of Leavitt path algebras is often added to the aforementioned three, namely, the Jacobson algebra $K\langle X,Y \  | \  XY = 1 \rangle$, which arises as the Leavitt path algebra $L_K(\mathcal{T})$  of the 
Toeplitz graph $\mathcal{T} = \ \  \ \xymatrix{ \bullet \ar@(ul,dl) \ar[r]
& \bullet  }$.     The Jacobson algebra was shown to be \Bez \   by Gerritzen \cite{G}. 

With these four classes of Leavitt path algebras having been shown to be \Bez,  the conjecture and subsequent establishment of Theorem \ref{LpasareBezout}  followed plausibly.  
\hfill $\Box$ 
}
\end{remark}






\begin{nota}

{\rm A cycle $c$ in a directed graph $E$ is called a {\it source cycle} in case $|r^{-1}(v)| = 1$ for every $v\in c^0$.
 In other words, $c$ is a source cycle if the vertices in $c^0$ receive no edges other than those which are already in the cycle $c$.  

 A subset $V$ of $E$ is called a \emph{source} if either $V=\{v\}$ where $v$ is a source vertex in $E$, or $V$ is the set of the vertices $c^0$ of a source cycle $c$ in $E$.   In case $E$ is finite, we denote by 
\begin{itemize}
\item $W$ the set of vertices $E^0 \setminus V$;
\item $\nu$ the sum $\sum_{v\in V}v$;
\item $\omega$ the sum  $\sum_{w\in W}w$.
\end{itemize} 
 Note that $\nu + \omega = 1_{L_K(E)}$.}
\end{nota}

\begin{remark}
{\rm Source cycles and source vertices share a number of properties, but differ in many ways  as well.  In our first few results we will be able to treat these two types of configurations simultaneously; but further on, we will need to distinguish them one from the other. }   \hfill $\Box$ 
\end{remark}

 From now on, if not differently specified, by $E$ we denote a finite graph, by $V\subseteq E$ a source and by $W$ the set of vertices $E\setminus V$.
By the {\it source elimination graph} $E\setminus V$ we mean the subgraph of $E$ from which we have eliminated all of the vertices in $V$, and all of the edges having source vertices in $V$.   We denote this  graph by $E_W$, and view it as  the ``restriction" subgraph of $E$ to $W$.   We note that $W$ is a hereditary subset of $E^0$ (because $V$ is a source).  
 

 \begin{lemma}
 Let $V$ be a source in $E$.    Then $\omega L_K(E) \omega = L_K(E_W)$.     More generally, let $U$ be any hereditary subset of $E^0$, and let $\mu = \sum_{u \in U}u$.  Then $\mu L_K(E) \mu = L_K(E_U)$.   
 \end{lemma}
 \begin{proof}
 In a nonzero expression of the form $\gamma \delta^\ast \in \omega L_K(E) \omega$, we must have $s(\gamma)\in W $ and $ r(\delta^\ast)=s(\delta) \in W$.  But then all of the vertices appearing in $\gamma$ and $\delta$ are in $W$ (as $V$ is a source and hence $W$ is hereditary).     The general statement follows similarly.  
 \end{proof}

\begin{nota}\label{JandDeltanotation}
{\rm 
We denote by $J$ the two-sided ideal of $L_K(E)$ generated by $\omega$.  That is,
$$J = L_K(E) \omega L_K(E)  = \langle w | w \in W \rangle .$$ 
If  $V = \{v\}$ is a non-isolated source vertex, then $J = L_K(E)$.   This follows from the (CK2) relation:  $v = \sum_{e\in s^{-1}(v)}ee^\ast$, and each $ee^\ast = e \cdot r(e)\cdot e^\ast$ is in $J$ because $r(e)\in W$.    In case $v$ is isolated, or $V$ is a source cycle,  we get $J \neq L_K(E)$.
We denote by $\Delta = \Delta(V)$ the set of edges of $E$ having source vertex in $V$ and range vertex in $W$:
$$\Delta = \Delta(V) = \{ e\in E^1 \  | \  s(e) \in V, r(e) \in W\}.$$ 
Because each element of $\Delta(V)$ is an edge, we get that $f_i^* f_j = 0$ for $f_i \neq f_j \in \Delta(V)$.  
If $V$ is the set of vertices of a  source cycle, we fix a starting vertex $v_1$,  so that  the source cycle will be denoted by 
\[c = e_1 e_2 \cdots e_n\quad\text{with }s(e_i) = v_i\text{ and }r(e_i) = v_{i+1} \ ( \mbox{for }1\leq i \leq n-1), \mbox{ and }  r(e_n) = s(v_1).\]

\noindent 
In such a case, we denote by  
$$\Theta = \Theta(V)$$   the (countable) set of all  paths  $\epsilon$ which have both  $s(\epsilon) \in V$ and $r(\epsilon) \in V$.     (If $v$ is a source vertex, we define $\Theta(v) = \{v\}$.)     \hfill $\Box$ 
}
\end{nota}

 If $p$ is a path in $E$ with $s(p)\in V$ and $r(p) \in W$, then $p$ may be completely and uniquely described  as: 
$$p = \epsilon \delta \beta,$$
where $\epsilon \in \Theta(V)$, $\delta \in \Delta(V)$, and $\beta$ is a path in $E_W$.  
We state this more formally, and thereby obtain  a useful description of the elements of $\nu L_K(E) \omega$.  

\begin{lemma}\label{nuL(E)omega}
Any $x\in \nu L_K(E)\omega$ can be written as
\[x=\sum_j \epsilon_{j} f_j \omega x_j \omega\]
where $\epsilon_j\in \Theta(V)$, $f_j\in \Delta(V)$ and $x_j\in L_K(E)$. 
\end{lemma}
\begin{proof}
This follows since any element $x \in \nu L_K(E) \omega$ can be written as  
 $\Sigma_j \gamma_j\delta_j^*$, where the $\gamma_j$'s are real paths such that $s(\gamma_j)\in V$ and  $r(\gamma_j)=r(\delta_j)$. Since $r(\delta_j^*)=s(\delta_j)\in W$ and  $W$ is hereditary, then  $r(\delta_j)=r(\gamma_j)\in W$, so we can describe the $\gamma_j$'s using  the previous discussion.  (We note  that  we can assume $\epsilon_{j} f_j\neq \epsilon_{\ell} f_\ell$ for $j\neq \ell$, and that any  coefficients in $K$  have been absorbed in the $x_j$ expressions.)     \end{proof}

%
%
 \begin{lemma}\label{rem:ortogonali}
Consider two paths $\epsilon_1f_1$ and $\epsilon_2f_2$, with $\epsilon_i\in \Theta(V)$, $f_i\in \Delta(V)$, and assume  $\epsilon_1f_1\neq \epsilon_2f_2$. Then $f_1^*\epsilon_1^* \epsilon_2f_2=0$.
\end{lemma} 
{\it Proof.}  Indeed if $\epsilon_1f_1\neq \epsilon_2f_2$ and $\epsilon_1=\epsilon_2$ then $f_1\not=f_2$ and hence
\[f_1^*\epsilon_1^* \epsilon_2f_2=f_1^*f_2=0.\]
On the other hand, assume $\epsilon_1\not=\epsilon_2$; if $\epsilon_1^* \epsilon_2\not=0$ then either $\epsilon_1=\epsilon_2\cdot\epsilon'$, or $\epsilon_2=\epsilon_1\cdot\epsilon'$.
In the first case we get
$\epsilon_1^* \epsilon_2f_2=(\epsilon')^* f_2=0,$ while 
in the second case $f_1^*\epsilon_1^* \epsilon_2=f_1^\ast\epsilon'=0.$ \hfill $\Box$

 \smallskip
 
\begin{lemma}\label{lemma:idempotent}  If $\gamma$ is any path in $E$, and $R = L_K(E)$, then:
\begin{enumerate}
\item $ \gamma \gamma^*$ is an idempotent in $R$;
\item $  \gamma \gamma^* \gamma = \gamma$ and   $\gamma^*\gamma\gamma^*=\gamma^*$;
\item $   R \gamma^* = R \gamma \gamma^*$  as left ideals of $R$;
\item $  R \gamma^* \cong Rw$  as left R-modules,   where $w = r(\gamma)$.    More generally, if $y \in R$, then $  R y\gamma^* \cong Ryw$.   
\end{enumerate}
\end{lemma}
\begin{proof}
$(1)$ are $(2)$ and trivial. $(3)$ follows from $(2)$.
The isomorphism in $(4)$ is established by considering  the two morphisms $Ryw\to Ry\gamma^*$ via $ryw\mapsto ryw\gamma^*$, and $Ry\gamma^* \to Ryw$ via $ry\gamma^* \mapsto  (ry\gamma^*)\gamma=   ryw$.  
\end{proof}
 
 A description of the two-sided ideal $J$ as a left ideal will be central to the discussion.
 
\begin{lemma}\label{Jasleftideal}
As a left ideal of $R = L_K(E)$,  
$$J=R \omega  \oplus (\oplus_{f\in \Delta(V); \epsilon \in \Theta }Rf^*\epsilon^*).$$
\end{lemma}
\begin{proof}
We know that $J= \langle w \ | \ w\notin V \rangle = R\omega R$  as a two-sided ideal. Since $R=R\omega\oplus R\nu$ and $R\omega\leq J$, we have $J=R\omega\oplus (R\nu \cap J)$. 
Let $m\in R\nu\cap J$; since $m\in J$, then  $m=(\sum \tau_i \kappa_i^* ) \omega (\sum \gamma_i\delta^*_i )$. 
Since $s(\gamma_i)\in W$ and $W$ is hereditary, also $r(\gamma_i)=s(\delta_i^*)=r(\delta_i)$ belongs to $W$. Since $m\in R\nu$, then $r(\delta_i^*)=s(\delta_i)\in V$; by Lemma~\ref{nuL(E)omega} we have
\[\delta^\ast_i=(\epsilon_{i} f_i \omega x_i\omega)^*=\omega x_i^* \omega  f_i^\ast\epsilon_i^\ast.\]

  So  $m$ belongs to  $\displaystyle\sum_i Rf_i^* \epsilon_i^*=\displaystyle\sum_i R \epsilon_i f_i f_i^* \epsilon_i^*$, where each $R \epsilon_i f_i f_i^*\epsilon_i^*$ is a direct summand of $_RR$ since its generator $\epsilon_i f_i f_i^* \epsilon_i^*$ is  an idempotent  by Lemma~\ref{lemma:idempotent}. Moreover, by Lemma~\ref{rem:ortogonali},  $\epsilon_i f_i f_i^* \epsilon_i^*$ is orthogonal to $\epsilon_j f_j f_j^* \epsilon_j^*$ unless $f_i=f_j$ and  $\epsilon_i=\epsilon_j $. 
This establishes that $J$ is contained in the displayed direct sum.  

The reverse containment is clear, since $f^* = \omega f^*$ for any $f\in \Delta(V)$ (as $r(f)\in W$ by definition).  
\end{proof}

We note that it is possible for the summand $\oplus_{f\in \Delta(V); \epsilon\in \Theta}Rf^*\epsilon^*$ which appears in the previous result to be zero; this happens in case $V$ is isolated, in which case $\Delta(V)$ is empty.

\smallskip

The general program here is to establish that $L_K(E)$ is \Bez \ by showing that the \Bez  \ property passes from various factor rings and subalgebras of $L_K(E)$ back to $L_K(E)$.  First of all, we record a useful observation which will facilitate passing the \Bez \ property to a ring  from its subrings.

\begin{lemma}\label{subringBezout}
Let $R$ be any ring, and    let $x_1, x_2,  \dots , x_n \in R$.   Let $S$ be a unital subring of $R$.  (If $R$ is unital, we do {\rm  not} require that $1_R = 1_S$).   Suppose  $\{x_1, x_2,  \dots , x_n \} \subseteq S,$ and suppose the left $S$-ideal $Sx_1 + Sx_2 + \cdots + Sx_n$ is principal; i.e., suppose there exists $x\in S$ for which $Sx_1 + Sx_2 + \cdots + Sx_n = Sx$.    Then $Rx_1 + Rx_2 + \cdots + Rx_n = Rx$.   In particular, if every finite subset of $R$ is contained in a unital  \Bez \ subring of $R$, then $R$ is \Bez. 
\end{lemma}
\begin{proof}
Since  $1_S x_i = x_i$ for all $1\leq i \leq n$, we have in particular that each $x_i$ is in $Sx_1 + Sx_2 + \cdots + Sx_n$, so we get that for each $i$ there exists $s_i \in S$ with $x_i = s_i x$.     But then $Rx_1 + Rx_2 + \cdots + Rx_n = Rs_1x + Rs_2x + \cdots + Rs_nx \subseteq Rx$.    The reverse containment follows as well:  $ x = 1_S x \in Sx$, so $x\in  Sx_1 + Sx_2 + \cdots + Sx_n \subseteq Rx_1 + Rx_2 + \cdots + Rx_n$.   
\end{proof}


\begin{proposition}\label{idealsinRomega}
Let $I$ be a finitely generated left ideal of $R=L_K(E)$ with $I \leq R\omega$.   Suppose $\omega R \omega$ is \Bez.   Then $I = Ry$ for some $y\in \omega R \omega$.  Moreover, there exists a split epimorphism  $R\omega \to I \to0$. 
\end{proposition}  

{\bf Proof.}   Write  $I = R r_1 \omega + \cdots + R r_m \omega$ for $r_1, \dots, r_m \in R$.  
 For each $1\leq i \leq m$, use Lemma \ref{nuL(E)omega}  to write  $r_i \omega = (\nu + \omega) r_i \omega = \nu r_i \omega + \omega r_i \omega =   ( \sum_j \epsilon_j f_j \omega r_j^i \omega ) + \omega r_i \omega$ for some $r_j^i\in R$.

Let $S=\omega R \omega$  and consider  the finitely generated  left $S$-ideal $\sum_i S  \omega r_i \omega + \sum_{i, j} S\omega r_j^i \omega $. 
 By the \Bez \ hypothesis on $S$,  there exists $y\in S$ such that $\sum_i S  \omega r_i \omega + \sum_{i, j} S\omega r_j^i \omega  =Sy$.  So, by  Lemma~\ref{subringBezout}, $\sum_i R  \omega r_i \omega + \sum_{i, j} R\omega r_j^i \omega = Ry$
and hence $I\leq Ry$.

 Conversely, $\omega r_i\omega \in I$ and $\nu r_i \omega \in I$ (since $r_i \omega \in I$).  And $\omega r_j^i \omega=f_j^*\epsilon_j^* \nu r_i\omega$ (this follows since if  $j\neq k$ then $\epsilon_jf_j\neq \epsilon_k f_k$ and so $f^*_j(\epsilon_j^*)\epsilon_kf_k=0$) . So $\omega r_j^i\omega \in I$. Since $y\in  \sum_i R  \omega r_i \omega + \sum_{i, j} R\omega r_j^i \omega $, we get $Ry \leq  I$.
 
 Summarizing, we have shown that if $I$ is a finitely generated left ideal of $R$ which is contained in $R\omega$, then $I = Ry$ for some  $y\in \omega R \omega = S$.  
 
 Since $R = L_K(E)$  is hereditary, necessarily $Ry$ is projective, so the epimorphism $\varphi: R\to Ry\to 0$  (sending $1\mapsto y$) splits. Notice that since $\nu y=0$ we get that $\varphi_{|_{R\nu}}=0$ and so we have a split epimorphism $R\omega \to Ry\to 0$.   \hfill $\Box$
 



 



\begin{lemma}\label{lemmaforidealsinRdeltastar}
Let $R$ denote $L_K(E)$, and let $V$ be a source in $E$.  Suppose  $\omega R \omega $ is \Bez.  Let $B\leq R \delta^*$ be a finitely generated left ideal, where $\delta$ is a real path  with $w = r(\delta)\in W$.   Then $B$ is principal; indeed, $B = Ry\delta^*$ for some $y\in \omega R \omega$.
\end{lemma}
\begin{proof}
By Lemma \ref{lemma:idempotent}(4) we have that the map $R\delta^* \to Rw$ given by right multiplication by $\delta$ is an isomorphism of left $R$-modules, so that $(B)\delta$ is a finitely generated  submodule of $Rw$; but $Rw \leq R\omega$, so $(B)\delta$ is isomorphic to a left ideal of $R$ contained in $R\omega$.   So Proposition \ref{idealsinRomega} applies to $(B)\delta$, to yield $(B)\delta = Ry $ for some $y\in \omega R \omega$.   Then $ (B)\delta \delta^* = (Ry)\delta^* = Ry\delta^*$; but $(B)\delta \delta^* = B$ because $B \leq R \delta^*$.
\end{proof}

We now  establish a result similar in flavor to Proposition \ref{idealsinRomega}, but for ideals contained in the other type of summands of $J$.    \begin{nota}\label{treenotation}
{\rm  Let $E$ be a graph, and $v\in E^0$.   The {\it tree of} $v$, denoted $T(v)$, is the smallest hereditary subset of $E^0$ containing $v$.  Formally,  
 $$T(v) = \{ u \in E^0 \ | \ \mbox{there exists a path } p \mbox{ in } E \mbox{ having } s(p) = v \mbox{ and } r(p) = u\}.$$
 }
  \end{nota}
 
\begin{proposition}\label{idealsinRdeltastar}
 Let $R$ denote $L_K(E)$, and let $V$ be a source in $E$.   Let  $B\leq R \delta^*$, where $\delta$ is a real path  with $r(\delta)\in W$.   Set
 $$W' := T(r(\delta)), \mbox{ and define } \omega' := \sum_{u\in W'}u.$$  
     Suppose $\omega R \omega$ and $\omega' R \omega'$ are \Bez.  Then $B$ is isomorphic to a direct summand of $R\omega'$.  Indeed, there is a split epimorphism $R\omega' \to B \to 0$.  
\end{proposition}
\begin{proof}
By Lemma \ref{lemmaforidealsinRdeltastar} we have $B = Ry\delta^*$ for some $y\in  \omega R \omega$.  Since $\omega \delta^* =w\delta^*= \omega' \delta^*$, without loss of generality we may assume $y \in \omega R \omega'$.    We define $\omega'' =  \omega - \omega'$,  so that  $\omega''\in W'' := W\setminus W'$.  Thus  $y=\omega' y\omega'+\omega'' y\omega'$ and $Ry=R(\omega' y\omega')+R(\omega'' y\omega')$, since $\omega'y\omega'=\omega'y$ and $y-\omega'y=\omega'' y\omega'\in Ry$. 

 Since the set $W'$ is hereditary, any element of the form $w'' r w' \in w'' R w'$  must be a sum of elements of the form $k \gamma \alpha^*$, where $k\in K$ and necessarily $\gamma$ is a path starting in $W''$ and ending in $W'$.    (Specifically, no summand of $w'' r w' $ can start with a ghost path.)   
So in particular we can write  $$\omega'' y\omega' =  \gamma_1 \omega'y_1\omega'+\cdots + \gamma_\ell \omega'y_\ell \omega',$$
where the $\gamma_i$ are distinct, each $\gamma_i = \kappa_i g_i$ for some  real path (possibly a vertex) $\kappa_i$, and edge  $g_i$, for which $s(\kappa_i)\in W''$, $r(\kappa_i)= s(g_i)\in W''$, and $r(g_i) \in W'$.     So the displayed equation gives  that $R(\omega''y\omega')\leq R(\omega' y_1 \omega')+\cdots+R(\omega' y_\ell \omega')$. 


For any two paths $\lambda_1, \lambda_2$ in $E$, the expression $\lambda_1^* \lambda_2$ is nonzero only when $\lambda_1$ is an initial subpath of $\lambda_2$, or $\lambda_2$ is an initial subpath of $\lambda_1$.   But because each  $\gamma_i$ has its final edge having source in $W'' $ and range in $W'$, the only way that some $\gamma_i$ can be an initial subpath of some $\gamma_j$ is for $\gamma_i = \gamma_j$.     Thus  $\gamma_i^* \gamma_j = 0$ for $i\neq j$.   
Using this,  and multiplying the displayed equation on the left by $\gamma_i^*$, yields 
  $\gamma_i^* \omega'' y \omega'=\omega' y_i \omega'$. Thus we get that $R(\omega' y_1 \omega')+\cdots+ R(\omega' y_\ell  \omega')\leq R\omega''y\omega'$. Hence  $R(\omega' y_1 \omega')+\cdots+ R(\omega' y_\ell \omega') =  R\omega''y\omega'$.  But then 
  $$Ry=R(\omega' y\omega')+R(\omega'' y\omega')=R(\omega' y\omega')+R(\omega' y_1 \omega')+\cdots+ R(\omega' y_\ell  \omega').$$  So we conclude that  $$ Z:=    \{ \omega' y \omega'\} \cup \{ \omega' y_i \omega' \ | \ 1\leq i \leq \ell   \}$$ is a set of generators for $Ry$ contained in $\omega' R \omega'$. But $\omega' R \omega'$ is \Bez \ by assumption.    So by Lemma \ref{subringBezout} we get that $Ry = Rx$ for some 
 $x \in \omega' R \omega'$. 

   
   In particular,  the map $\varphi: R\omega' \to Rx\delta^* = B$ given by $r \omega' \mapsto r \omega' x \delta^* = r  x \delta^* $ is an epimorphism.   Since $B$ is a left ideal of $R$ it is necessarily projective.  The result follows.  
\end{proof}

The ``exotically-behaved" algebra $L_K(1,n)$ for $n\geq 2$  provides a model for the behavior of a large class of algebras, each of which will be ``easily" shown to be \Bez.  

\begin{definition}\label{UGNDef}
{\rm A ring $R$ is said to have the {\it Unbounded Generating  Number} property (more efficiently, {\it UGN}) in case for each pair of positive integers $m,n$,  the existence of an epimorphism of  left $R$-modules $R^m \to R^n \to 0$ implies  $m\geq n$.   }
\end{definition}   

\begin{lemma}\label{notUGNgivesBez}
Suppose $R$ is a ring for which there is a left $R$-module epimorphism $R \to R^2$.    (In other words, suppose $R$ fails the UGN property for the pair $m=1, n=2$.)   Then $R$ is \Bez.
\end{lemma}
\begin{proof}
Easily the given condition implies that there is a left $R$-module epimorphism $R \to R^n$ for every positive integer $n$.    If $I$ is a finitely generated left ideal of $R$, then there exists an epimorphism $R^n \to I$, and so there exists an epimorphism $R \to I$, so that $I$ is principal.  
\end{proof}

We note that for any ring $R$ satisfying the hypotheses of Lemma \ref{notUGNgivesBez}, we may indeed conclude that every finitely generated left $R$-{\it module} is principal.   

\medskip

Let $E$ be any graph, and $v\in E^0$ a non-sink.  We consider the nonempty set $s^{-1}(v)$, i.e., the set of edges emitted by $v$.    For  any nontrivial  partition $\mathcal{P}$ of $s^{-1}(v)$, one may define the {\it out-split graph of $E$ at $v$ relative to $\mathcal{P}$}, denoted  by $E_v(\mathcal{P})$.   (This definition  is  a specific case of a more general construction;  although the general construction does not concern us here, we will utilize this one specific case of it.)      In particular,  suppose $v$ is a source vertex for which $|s^{-1}(v)| = \{f_1, ..., f_n\}$ for some $n\geq 2$.   Let $\mathcal{P}$ be the partition $\{f_1, ... , f_{n-1}\} \sqcup \{f_n\}$ of $s^{-1}(v)$.   Then  $E_v(\mathcal{P})$ may be described in words as     the graph whose:    vertex set otherwise looks exactly like $E^0$,  but with one additional vertex $\hat{v}$ added;  and whose edge set otherwise  looks exactly like $E^1$, except with $f_n$  eliminated, and with a new edge $g$  added, for which $s(g) = \hat{v}$ and $r(g) = r(f_n)$.     Even less formally,  we build $E_v(\mathcal{P})$ by giving the edge $f_n$ its own new source vertex, and leaving everything else alone.   (For additional information about the out-split process, see e.g. \cite{AALP}.) 

We will denote $E_v(\mathcal{P})$ by $\Gamma$ in the sequel.  To distinguish between the two related graphs $E$ and $\Gamma$, we denote by $v_0$ the vertex $v$ when viewed in $\Gamma^0$.    

\begin{theorem}\label{outsplittheorem}  \cite[Theorem 2.8]{AALP}  In the situation described above, there is an isomorphism of $K$-algebras
$$L_K(E) \cong  L_K(E_v(\mathcal{P})).$$
\end{theorem}

We are now in position to prove our main result.

\begin{theorem}\label{LpasareBezout}
Let $E$ be a finite graph and $K$ any field.  Then $L_K(E)$ is \Bez.  
\end{theorem}

\begin{proof}
We proceed by induction on $|E^0|$, the number of vertices in $E$.   

\medskip

If $|E^0| = 1,$ then there are three possibilities:  either $E$ has no edges, exactly one edge, or $m\geq 2$  edges.    The Leavitt path algebras of such graphs are then, respectively,  $K$, $K[x,x^{-1}]$, and $L_K(1,m)$.   But as established in the introduction, each of these $K$-algebras is \Bez, thus establishing the base case.

\medskip

So we assume that $|E^0| = n,$  and that $L_K(F)$ is \Bez \ for any graph $F$ having fewer than $n$ vertices.  We show that $L_K(E)$ is \Bez.    As in the $|E^0| = 1$ case, there are three possibilities to analyze.

\bigskip

\underline{ Case 1:  $E$ contains neither source vertices nor source cycles.}  

\smallskip

\ \ \ \ \  By \cite[Theorem 3.9]{ANP}, $L_K(E)$ does not have Unbounded Generating Number.   By \cite[Remark 3.10]{ANP}, the failure of UGN for a Leavitt path algebra implies that there is an epimorphism of left $L_K(E)$-modules $L_K(E) \to L_K(E)^2$.    So by Lemma \ref{notUGNgivesBez} $L_K(E)$ is \Bez,  thus establishing Case 1.

\bigskip

\underline{ Case 2:  $E$ contains a source vertex.}  

\smallskip

\ \ \ \ \  Let $v$ denote a source vertex in $E$.         If $v$ is also a sink (i.e., if $v$ is an isolated vertex), then it is easy to establish that  $L_K(E)$ is isomorphic as $K$-algebras to the ring direct sum 
$ K \oplus L_K(E \setminus \{v\})$.  
 Clearly $K$ is \Bez,  and $L_K(E \setminus \{v\})$ is \Bez \ by the induction hypothesis, so that $L_K(E)$ is a direct sum of \Bez  \ rings, and hence is \Bez.  

 We now consider the case where $|s^{-1}(v)| = 1$, i.e., that $v$ emits a single edge, which we denote by $f$.   (We will subsequently analyze the general case.)     Let $w$ denote $r(f)$.   

So let $I \leq L_K(E)$ be a finitely generated left ideal.  Then $I \leq L_K(E)=Rv \oplus R\omega$.  Note that $Rv = Rf^*$,  using   $s^{-1}(v) = \{f\}$  (which yields  $ff^* = v$ by the CK2 relation), and using $f^*ff^* = f^*$.    This yields in particular that $Rv \cong Rw$, by Lemma \ref{lemma:idempotent}(4).  

Then projecting $I$ into $Rf^*$ along $R\omega$  gives that  $I / I \cap R\omega$ is isomorphic to a (finitely generated) submodule of $Rf^* = Rv$, so that $I / I \cap R\omega$ is isomorphic to a left ideal of $R$ contained in $Rf^*$; we denote this left ideal of $R$ by $B$.    As $R$ is hereditary, this gives in particular  that $I / I \cap R\omega$ is projective.    Now consider the exact sequence
$$0 \to I \cap R\omega \to I  \to   I / I \cap R\omega  \to  0.$$
Since $I / I \cap R\omega$ is projective, the sequence splits, and we have 
$$ I \cong (I \cap R\omega) \oplus I / I \cap R\omega \cong (I \cap R\omega) \oplus B.$$


 

\noindent 
   Let $W'$ denote $T(w)$, and $\omega'$ the sum over all vertices in $W'$.   
   By the induction hypothesis, $\omega R\omega$ and $\omega' R \omega'$ are \Bez.   So $B$  satisfies the conditions of Proposition \ref{idealsinRdeltastar}, which allows us to write $B =   Ryf^*$ for some $y\in \omega' R \omega'.$  So we have
 $$ I \cong (I \cap  R\omega) \oplus Ryf^*.$$
 Since $wf^* = f^*$, we have $Ryf^* = Rywf^*$, so we may assume without loss that $yw = y$.   As well, since $y \in \omega' R \omega'$, we have $vy = 0$.   
 
 By Lemma \ref{lemma:idempotent}(4), $Ryf^* \cong Ryw = Ry$.

 Now we analyze the summand $I \cap R\omega$.   Since $I$ is finitely generated, and $I \cap R\omega$ is a direct summand of $I$, then $I\cap R\omega$ is finitely generated as well.  Since we also have $I\cap R\omega \leq R\omega$, and $\omega R \omega$ is \Bez \ by the induction hypothesis, Proposition \ref{idealsinRomega} applies to $N:= I \cap R\omega$ to yield that $N$ is isomorphic to a direct summand of $R\omega$.    We write $\widehat{W} = W \setminus \{w\}$, and let $\widehat{\omega} = \omega - w$.   Then $R\omega = Rw \oplus R\widehat{\omega}$.   So we have 
$$ I \cong  N \oplus Ryf^* \cong N\oplus Ry, \ \mbox{which is isomorphic to a direct summand of } (Rw \oplus R\widehat{\omega}) \oplus Ry.$$
 
 Now consider the left {\it ideal} $R\widehat{\omega} \oplus Ry$  (this is a subideal of $R\omega$, since $y=yw$.)  Let $S$ denote $\omega R \omega$.   We have $y = \omega' y w \in S$, and $\widehat{\omega} \in S$.  Note that $Sy \cap S\widehat{\omega} = 0$, so that $Sy \oplus S\widehat{\omega}$ is a (finitely generated) left ideal of $S$.   But $S$ is \Bez \ by the induction hypothesis, so   $Sy \oplus S\widehat{\omega} = Sx$ for some $x\in S$.   Then Lemma \ref{subringBezout}  
  yields that $Ry \oplus R\widehat{\omega} = Rx$.    But $Rx \leq R\omega$.  So by Proposition \ref{idealsinRomega}, $Rx$ is isomorphic to a direct summand of $R\omega$, so that  $Ry \oplus R\widehat{\omega}$ is isomorphic to a direct summand of $R\omega$.   So by the previous display,  $I$ is isomorphic to a direct summand of $Rw \oplus (R\widehat{\omega} \oplus Ry)$, which is then isomorphic to a direct summand of $Rw \oplus R\omega$, which as noted previously is isomorphic to $Rv \oplus R\omega$, which is precisely $R$ itself.  So $I$ is isomorphic to a direct summand of $R$, and hence is principal.   
 
 So we have shown:  If $E$ has the property that $L_K(F)$ is \Bez \ for any graph $F$ having fewer vertices than $E$, and $E$ contains a source vertex $v$ for which $|s^{-1}(v)| = 1$, then $L_K(E)$ is \Bez.

 \smallskip

 We now consider the general case, in which we assume that $E$ has the property that $L_K(F)$ is \Bez \ for any graph $F$ having fewer vertices than $E$, and $E$ contains a source vertex $v$ for which $|s^{-1}(v)| = m$ for some positive integer $m$.   We seek to show that $L_K(E)$ is \Bez.    We proceed  by induction on $m$.  The $m=1$ case is precisely the case treated in the previous discussion.       So we assume that if  $G$ is any graph which has the property that $L_K(F)$ is \Bez \ for any graph $F$ having fewer vertices than $G$, and $G$ contains a source vertex $v$ for which $|s^{-1}(v)| = m-1$, then $L_K(G)$ is \Bez.  We show that the result passes to the graph $E$.  
 We denote the set $s^{-1}(v)$  by $ \{f_1,  ..., f_{m-1}, f_m\}$.  
 
 We construct the out-split graph $\Gamma =   E_v(\mathcal{P})$ for the partition $ \mathcal{P} = \{f_1,  ..., f_{m-1}\} \sqcup \{f_m\}$ of $s^{-1}(v)$.    So (using the description given prior to Theorem \ref{outsplittheorem}), the vertex $v_0$ of $\Gamma$ is a source vertex, with $|s^{-1}(v_0) |= m-1$.    Moreover, the ``new" vertex $\hat{v}$ of $\Gamma$ is a source vertex with $|s^{-1}(\hat{v})| = 1$; we write $s^{-1}(\hat{v}) = \{g\}$.    Note that $|\Gamma^0| = |E^0| + 1$.   
 
  Using the isomorphism provided by Theorem \ref{outsplittheorem}, it suffices to show that $L_K(\Gamma)$ is \Bez.   We define the graph $\overline{\Gamma}$ to be the graph $\Gamma$, with the vertex $\hat{v}$ and edge $g$ eliminated.   So $|\overline{\Gamma}^0| = |E^0|$ and $|s^{-1}(v_0) |<m$.   Then $\overline{\Gamma}$ satisfies the appropriate criteria given in the previous two paragraphs, so we conclude that   $L_K(\overline{\Gamma})$ is \Bez.     
  
  
  
  
 So  $\Gamma$ is a graph which contains a source vertex $\hat{v}$ which emits a single edge $g$.   We re-examine the proof of the $n=1$ case above, but replacing $v$ and $f$ by the pair $\hat{v}$ and $g$.     We note that  $T(\hat{v}) \subseteq T(v)$  and $T(r(f))=T(r(g))$, so that at the point in the previous proof where we construct $\omega'$, we may still in the new proof conclude that $\omega' L_K(\Gamma) \omega'$ is \Bez \  (since $|W'|\leq |W|<|E^0|$).    On the other hand,    at the point in the previous proof where we use that $\omega L_K(E) \omega$ is \Bez \ by the induction hypothesis,  we now have that $   \omega L_K(\Gamma) \omega \cong L_K(\overline{\Gamma})$ is \Bez \ by the argument in the previous paragraph.      So proceeding as in the $n=1$ case, we conclude that $L_K(\Gamma) \cong L_K(E)$ is \Bez, as desired. 
 
 \bigskip

\underline{ Case 3:  $E$ contains a source cycle.}  

\smallskip

 \ \ \ \ \  The third and final case in the induction  involves an analysis of the situation in which $E$  contains a source cycle.    Recall that we have denoted the cycle by $c$, and the set of vertices $c^0$ by $V$, and the set of vertices $E^0 \setminus V$ by $W$.  As in Case~2, if $c$ is an isolated cycle, then $L_K(E)$ is isomorphic as $K$-algebra to the ring direct sum $L_K(E_V)\oplus L_K(E_W)$.   The ring $L_K(E_V)$ is \Bez , since it is isomorphic to $M_n(K[x, x^{-1}])$ with $n$ the length of the cycle (see Remark~\ref{Bezoutremark} together with \cite[Lemma 2.7.1]{TheBook}), and $L_K(E_W)$ is \Bez \  by the induction hypothesis. So $L_K(E)$ is \Bez.  
 
 If $c$ is not isolated, then the two-sided ideal $ \langle W \rangle$ of $L_K(E)$ is denoted by $J$.   By Lemma \ref{Jasleftideal}, $$J=R \omega  \oplus (\oplus_{f\in \Delta(V); \epsilon \in \Theta(V) }Rf^*\epsilon^*).$$    Since $c$ is a cycle, the set $\Theta(V)$ is infinite (for instance, it contains $c^n$ for each $n\in \N$).   
 
  We will use a two-step process to show that $L_K(E)$ is \Bez \ in this situation as well.     First, we establish that any finitely generated left ideal $I$ of $L_K(E)$ which is contained in $J$ is principal.   We then parlay this first step result to show that every  finitely generated left ideal of $L_K(E)$ is indeed principal.

For any $N\geq 0$, we define $\Theta_N(V)\subseteq \Theta(V)$ to be the (finite) subset of $\Theta (V)$ consisting of those paths of the form $\alpha c^i \beta$, where $i \leq N$, and neither $\alpha$ nor $\beta$ contain $c$ as a subpath.  
We then set 
$$J_{N}:=R \omega  \oplus (\oplus_{f\in \Delta(V); \epsilon\in \Theta_N(V)}Rf^*\epsilon^*).$$
Since $I \leq J$ is finitely generated, there necessarily exists $N \geq 0$ for which $I \leq J_N$.    This observations allows us to prove  that $I$ is principal by  using  induction on the number of summands which appear in the above display.  Call this number $m$.     

The $m=1$ case is established as follows.    By the induction hypothesis $\omega R \omega$ is \Bez.  If $I\leq R\omega$, then Proposition \ref{idealsinRomega} can be invoked.  On the other hand, if $I \leq Rf^*\epsilon^*$ for some $f$ and $\epsilon$, then Lemma \ref{lemmaforidealsinRdeltastar} with $\delta=\epsilon f$ finishes the job.

So we assume that  if $I  \leq A_1 \oplus A_2 \oplus \cdots \oplus A_{m-1}$  (where each $A_t$ is of the form $Rf^*\epsilon^*$, and possibly  $A_t = R\omega$ for some (unique) $t$), then $I$ is principal.    We show that if 
   $I  \leq A_1 \oplus A_2 \oplus \cdots \oplus A_m$, where each $A_t$ is of the form $Rf^*\epsilon^*$, and possibly  $A_t = R\omega$ for some (unique) $t$, then $I$ is principal as well.         If $I \cap A_t = \{0\}$ for some $t$  then $I$ is isomorphic (via  the projection along $A_t$)  to an ideal which is contained in a direct sum of $m-1$ $A_i's$,  and so we are done by induction.   Hence we may assume that $I\cap A_t \neq \{0\}$ for all $1\leq t \leq m$.     So in particular we have $I \cap A_j \neq \{0\}$ for some $A_j = Rf^* \epsilon^*$.

 Then $I / I \cap A_j \cong (I + A_j )/ A_j \leq A_1 \oplus \cdots \oplus A_{j-1} \oplus A_{j+1} \oplus \cdots \oplus A_m$.  So by induction we get $I / I \cap A_j$ is isomorphic to a principal left ideal; that is, since $R$ is hereditary and hence any left ideal is projective, $I / I \cap A_j$ is isomorphic to a direct summand of $R = R\nu \oplus R\omega$.

 
 
 Moreover, using again that  $I / I \cap A_j$ is projective,   the short exact sequence
 $$0 \to I \cap A_j \to I \to I / I\cap A_j \to 0$$
 splits, and so
 $$I \cong (I \cap A_j) \oplus I / I \cap A_j.$$
 So, as a direct summand of $I$,  in particular we get that $B:= I \cap A_j$ is both finitely generated and projective.  Let $W'$ denote $T(r(f))$, and let $\omega'$ denote the sum of the vertices in $W'$.  Then $\omega' R \omega'$ is \Bez \ by the induction hypothesis.    So we may apply Proposition \ref{idealsinRdeltastar} to get that $I \cap A_j$ is principal, and is isomorphic to a direct summand of $R\omega'$.   
 
 So we have 
  $$I \cong (I \cap A_j) \oplus I / I \cap A_j, \ \mbox{which is isomorphic to a direct summand of } R\omega' \oplus R\nu \oplus R\omega.$$

Now consider the cycle $c $ for which $c^0 = V$.   Denote $V$ by $v_1, v_2, \dots, v_\ell$, and the edges of $c$ by $h_1, h_2, ..., h_\ell$, with $s(h_i) = v_i$ and $r(h_i) = v_{i+1}$.  By the CK2 relation we have $v_1 = h_1 h_1^* + \sum_{f\in \Delta(V), s(f) = v_1} ff^*$.  So $Rv_1 \cong  Rh_1h_1^*  \oplus (\oplus_{f\in \Delta(V), s(f) = v_1} Rff^*)$.   By Lemma \ref{lemma:idempotent} parts (3) and (4)  we have $Rv_2 \cong Rh_1h_1^*$, so that $Rv_1 \cong Rv_2  \oplus (\oplus_{f\in \Delta(V), s(f) = v_1} Rff^* )$.  Using this idea at each subsequent vertex of the cycle, we eventually return back to $v_1$, and we get 
$$Rv_1 \cong Rv_1 \oplus (\oplus_{f\in \Delta(V)}Rff^*).$$
We denote $\oplus_{f\in \Delta(V)}Rff^*$ by $N$.   So we have shown that $Rv_1 \cong Rv_1 \oplus N$.   But by repeated substitution, this then gives $$Rv_1 \cong Rv_1 \oplus N^s$$ for any positive integer $s$.  

We use this same idea to establish that $Rv_i \cong Rv_i \oplus N^s$ for all $1\leq i \leq \ell$ and all positive integers $s$.   So  $$\oplus_{i=1}^\ell Rv_i \cong (\oplus_{i=1}^\ell Rv_i) \oplus N^{s\ell}$$
     for any positive integer $s$; rephrased,
     $$R\nu \cong R\nu \oplus N^{s\ell} \ \mbox{for any positive integer } s.$$

    Now consider any vertex $u\in T(V) \setminus V = W'$.   Then there is a path from $V$ to $u$, which necessarily must have as its initial edge an element $f\in \Delta(V)$.  So using the same argument as in the previous paragraphs, by the CK2 relations we have that, for each $u\in W'$,  $Ru$ is isomorphic to a direct summand of $Rff^*$ for some $f\in \Delta(V)$.  Let $p = |W'|$.   Then $\oplus_{u\in W'}Ru$ is isomorphic to a direct summand of $N^p$; that is, $R\omega' $ is isomorphic to a direct summand of $N^p$.  So 
    $R\omega' \oplus R\omega'$ is isomorphic to a direct summand of $N^{2p}$.  So 
    $R\nu \oplus R\omega' \oplus R\omega' $ is isomorphic to a direct summand of  
    $R\nu \oplus N^{2p}$, which is isomorphic to a direct summand of $R\nu \oplus N^{s\ell}$ for some $s$ sufficiently large.   But $R\nu \oplus N^{s\ell}$  in turn is isomorphic to  $R\nu$, by the previous display.   In other words,
    $$R\nu \oplus R\omega' \oplus R\omega' \ \mbox{is isomorphic to a direct summand of } R\nu.$$
    

    
       Now write $1_R = \nu + \omega = \nu +  \omega' + \omega''$.   Then as left-modules we have $I $ is isomorphic to a direct summand of $R\omega' \oplus R\omega  \oplus  R\nu $, so $I$ is isomorphic to a direct summand of $R\omega' \oplus R\omega' \oplus R\omega'' \oplus R\nu \cong R\omega' \oplus R\omega' \oplus R\nu \oplus R\omega'' $, which by the previous display gives that $I$ is isomorphic to a direct summand of $R\nu \oplus R\omega''$.   But this last module is indeed a left ideal (a direct summand of $R$), so that $I$ is isomorphic to a direct summand of $R$, i.e., $I$ is principal.
       
      \smallskip
     
 So we have established the first step of Case 3, namely,  that any finitely generated  left ideal $I$ which is contained in $J$ is principal.  
 
 \smallskip

  Next, we    establish the result for finitely generated left ideals $I$ which  contain $J$.    
  To do so, we  use some general results about Leavitt path algebras, see e.g. \cite{TheBook}.    We have that $L_K(E) / J \cong L_K( E/ H(W) )$, where $H(W)$ is the hereditary \emph{ saturated closure} of $W$ in $E$, and  $E/H(W)$ is the graph gotten from $E$ by eliminating all the vertices in $H(W)$ and all attendant edges.

      In particular, if $V$ is the set of vertices of a source cycle, then $L_K(E) / J \cong L_K(V)$, which it turns out to be isomorphic to $ {\rm M}_n(K[x,x^{-1}])$ as $K$-algebras. Let as usual $c=e_1\dots e_n$ be the source cycle based on $v_1$. Then    specifically, the isomorphism is given by
$$(p_{ij}(x)) \in  {\rm M}_n(K[x,x^{-1}]) \ \mapsto \ \sum_{i,j = 1}^n \alpha_i p_{ij}(c) \alpha_j^\ast$$
where $\alpha_i=e_i\cdots e_n$.
But ${\rm M}_n(K[x,x^{-1}])$ is a principal ideal ring (a proof is given in \cite[Theorem 40]{J}).       Specifically, the  left ideal $I / J$ of $L_K(E) / J$ is principal;  say $I / J = (L_K(E)/J)\overline{y}$, so that $I = L_K(E)y + J$ for some $$y = \sum_{i,j = 1}^n \alpha_i p_{ij}(c) \alpha_j^\ast \in R.$$  

\noindent
   For any $t \in \N$ let $d_t \in R$ denote the element $$d_t = \sum_{i=1}^n \alpha_i c^t \alpha_i^\ast.$$ Then a straightforward computation yields that $$d_t y =   \sum_{i,j = 1}^n \alpha_i c^t p_{ij}(c) \alpha_j^\ast.$$   Moreover, as left $R$-ideals,   $Rd_t y = Ry$, because (again by an easy computation) we have $d_t^\ast d_t y= y$. Indeed, $d_t^\ast d_t =\nu$ and $\nu y=y$.    So, if $y_t=d_ty$, we may write $I = L_K(E)y_t +J$, where $d_t d_t^* y_t = y_t$.
   
   
   \medskip
   
 Recall that   for each positive integer $M$ we have denoted by   $J_{M}$  the left $R$-ideal
$$J_{M} = R\omega \oplus (\oplus_{f\in \Delta(V); \epsilon \in \Theta_M}R \epsilon ff^* \epsilon^*)$$
$$ = R\omega \oplus (\oplus_{f\in \Delta(V); \epsilon \in \Theta_M }Rf^*\epsilon^*).$$
Let $e_{M}$ denote the idempotent   
$$e_{M } \ = \ \omega  + \sum_{f\in \Delta(V); \epsilon \in \Theta_M} \epsilon ff^* \epsilon^*.$$  As the set of summands appearing in $e_{M}$ is a set of orthogonal idempotents in $R$,  $e_{ M}$  itself is an idempotent. Clearly  $e_{ M}\in J_M$ and $J_{M}= L_K(E)e_{M}$.  
 
Since $I$ is finitely generated as a left $R$-module, we have that there exists $M$ for which $I \subseteq L_K(E)y_M + J_{M-1}$.     But as $L_K(E)y_M + J_{M-1} \subseteq L_K(E)y_M + J=I$, we have $$I = L_K(E)y_M + J_{M-1}=L_K(E)y_M + L_K(E)e_{M-1}.$$ 
Remember that  $d_{M} d_{M}^* y_M = y_M$;
on the other hand, 
\[
\begin{aligned}
d_{M}^*  e_{M-1} &=(\sum_i\alpha_i (c^*)^M\alpha_i^*)(\omega+\sum_{f\in \Delta(V); \epsilon\in\Theta_{M-1}}\epsilon f f^*\epsilon^*)\\
&=(\sum_i\alpha_i (c^*)^M\alpha_i^*)(\sum_{f\in \Delta(V); \epsilon\in\Theta_{M-1}}\epsilon f f^*\epsilon^*);
\end{aligned}\]
now by an easy computation we get $(c^*)^M\alpha_i^*\epsilon f =0$ for each $i$, $\epsilon$ and $f$, since $M-1$ is the maximal power of $c$ in any sequence $\epsilon=v_i\cdots c^j\cdots v_k\in\Theta_{M-1}$.

We finally claim that $L_K(E)y_M +L_K(E)e_{M-1} = L_K(E)(y_M+e_{M-1})$.  The inclusion  $\supseteq$ is clear. But $d_{M} d_{M}^\ast (y_M+ e_{M-1}) = y_M +0 =  y_M$ by the above computation, so $y_M\in L_K (E)(y_M+e_{M-1})$ and hence also $e_{M-1}$ belongs to $L_K (E)(y_M+e_{M-1})$.

 So we have established the second step of Case 3, namely,  that any finitely generated  left ideal $I$ of $R$ which is contained in $J$ is principal.   
 
 \smallskip

 We now use the two previously established steps   to verify  the result in the situation  where $I$ is not contained in $J$.  This will then complete the  proof of Case 3.

Since $I/I\cap J\cong (I+J)/J$ is a (finitely generated) ideal of the principal ideal ring $L_K(E)/J \cong M_n(K[x,x^{-1}])$,
%
 then  $I/I\cap J\cong (I+J)/J$ is cyclic as well, generated by an element $\overline{y}$, $y\in I$. So $I=Ry+I\cap J$: notice that $y=\nu y+\omega y$, where $\nu y\in I$ and $\omega y\in I\cap J$. And $I\leq R\nu y+R\omega y+I\cap J=R\nu y+I\cap J\leq I$ since $\nu y\in I$. Hence $I=R\nu y+I\cap J$.

 Notice that any element  $s\in J$ belongs to $Re_{M}$, for a suitable idempotent $e_{M}$ as defined before, so it is of the form $s=re_{M}=(\Sigma_i \gamma_i{\delta_i}^*){e_M}$.  We claim that there exists a suitable $N_k$ such that $(c_k^*)^{N_k}s=0$,  for any $k=1,\cdots, n$, where $c_k$ denotes the cycle based in $v_k$. It is not restrictive to assume $r$ to be a monomial $\gamma\delta^*$ and so consider the product $\gamma\delta^*e_M$ . First notice that if ${c_k^*}^t\gamma\neq 0$ for any $t\geq 0$, then it has to be $\gamma=c_k^{h}e_{k}\cdots e_{k'}$ for $h\geq 0$ and suitable $ k'$. And in such a case, since $V$ is a source,  it has to be $\delta\in \Theta$ and in particular $\delta^*=e_{k'}^* e^*_{k'-1}\cdots$. Distributing  $\gamma\delta^*$ over the summands of $e_M$ we get expressions $c_k^{h}e_{k}\cdots e_{k'}e_{k'}^*\cdots \epsilon f\cdots$, each of those involving $f\in \Delta(V)$ (or $\omega)$. So we can choose  a suitable $N_k$ such that $(c_k^*)^{N_k}\gamma\delta^*e_{M}=0$.

 Let now $x_j=s_j\nu y+t_j$, for $j=1\dots n$, be  a set of generators for $I$ with $t_j\in I\cap J$. Consider $M=Rt_1+\dots +Rt_n$. Then $M$ is a finitely generated ideal in $I\cap J$ and in particular it is contained in $J$. So, 
 now invoking  the previously established first step of Case 3, we have that $M$ is principal,   generated by an element $t\in J$.  Choose an $N$ such that ${ c_k^*}^Nt=0$, so that   $c_k^N{c_k^*}^Nt=0$, for any $k$. Call  $C^N=\Sigma_k (c_k^N)$ and ${C^*}^N=\Sigma_k ({c^*_k}^N)$. Then  ${C^*}^Nt=0$. Notice that ${C^*}^NC^N=v_1+\cdots+v_n=\nu$
 
Consider  now $I\leq R\nu y+Rt \leq R(C^N\nu y)+Rt$, since $\nu y={C^*}^N(C^N\nu y)$. We claim that $I=R(C^N\nu y+t)$.  Clearly  $I\geq R(C^N\nu y+t)$, since $t$ and $\nu y$ and so $t+C^N\nu y$ are in $I$. For the other inclusion note that $C^N(C^*)^N(C^N\nu y+t)=C^N\nu y+ C^N(C^*)^Nt=C^N\nu y$. So $C^N\nu y\in R(C^N\nu y+t)$ and hence  $t\in  R(C^N\nu y+t)$. Thus $RC^N\nu y+Rt=R(C^N\nu y+t)$  and so $I\leq R(C^N\nu y+t)$. 

\smallskip

In summary, we have shown that $I$ is a principal ideal whenever $I$ is a finitely generated ideal which is contained in $J$, and also whenever $I$ is a finitely generated ideal which is not contained in $J$.   So we have completed Case 3 in the induction, and thereby have established the Theorem. 
  \end{proof}


Almost the entirety  of the heavy lifting required to show that $L_K(E)$ is \Bez \ for {\it all} graphs $E$ has been completed by establishing the result for finite graphs, i.e.,  Theorem \ref{LpasareBezout}.

\begin{corollary}\label{LpasBezoutallgraphs}
Let $E$ be an arbitrary graph and $K$ any field.   Then $L_K(E)$ is \Bez.
\end{corollary}
\begin{proof}  By \cite[Corollary 1.6.10]{TheBook} (or by \cite[Theorem 4.1]{RangaRoozbeh}),  $L_K(E)$ is the direct limit of unital subalgebras, each of which is isomorphic to the Leavitt path $K$-algebra of a finite graph.  By Theorem \ref{LpasareBezout}, each of these unital subalgebras is \Bez.   So every finite set of elements of $L_K(E)$ is contained in a unital \Bez \ subring of $R$.  Now apply Lemma \ref{subringBezout}.
\end{proof}

To provide some comparison with the previously mentioned work on principal ideal rings, we offer the following.

\begin{proposition}\label{pirLpas}
Let $E$ be a finite graph.   Then $L_K(E)$ is a principal ideal ring if and only if no cycle in $E$ has an exit.   
\end{proposition}  

\begin{proof}
If no cycle in $E$ has an exit, then by \cite[Corollary 2.7.5]{TheBook} $L_K(E)$ is isomorphic as rings to a direct sum of copies of matrices over $K$ and/or matrices over $K[x,x^{-1}]$.   But each of these is well-known to be a principal ideal ring, hence so is their direct sum.

On the other hand, suppose there is a cycle in $E$ which has an exit.  Denote the cycle by $c$, and without loss of generality assume that $c$ is based at $v$, and that the presumed exit $f$ for the cycle has $s(f) = v$.    For each positive integer $i$ consider the element  $ \psi_i := c^i (c^*)^i - c^{i+1} (c^*)^{i+1} \in L_K(E)$.     A straightforward computation yields that each $\psi_i$ is an idempotent, and that $\psi_i \psi_j = 0$ for $i\neq j$.   Moreover, each $\psi_i$ is nonzero, since otherwise $ 0 = \psi_i = \psi_i c^i f = c^if - c^{i+1}c^* f = c^if - 0 = c^i f$, which contradicts the previously mentioned property that $KE$ embeds in $L_K(E)$ injectively.   So the left ideal $\oplus_{i\geq 1} L_K(E) \psi_i$ of $L_K(E)$  is not finitely generated, and therefore cannot   be principal.  
\end{proof}


As one nice consequence of Theorem \ref{LpasareBezout}, we get the following description of the finitely generated projective $L_K(E)$-modules which appear (up to isomorphism) as left ideals of $L_K(E)$.

\begin{corollary}\label{fingenprojwhichareleftideals}
Let $E$  be an arbitrary graph and $K$ any field.  Let $R = L_K(E)$.  Let $P$ be a finitely generated projective left $R$-module.  Then $P$ is isomorphic to a left ideal of $R$ if and only if $P$ is principal.
\end{corollary}
\begin{proof}
If $P$ is isomorphic to a left ideal of $R$ then Corollary~\ref{LpasBezoutallgraphs}  gives the result.   Conversely, if $P$ is principal then there exists $R \rightarrow P \rightarrow 0$, which splits since $P$ is projective, and so $P$ is isomorphic to a direct summand of $R$, and thus isomorphic to a left ideal of $R$. 
\end{proof}

As another consequence of our main result, we re-establish the two-sided version of the \Bez \ property for Leavitt path algebras.

\begin{corollary} \cite[Corollary 8]{Ranga}   Let $E$ be an arbitrary graph and $K$ any field.   Let $I$ be a two-sided ideal of $R = L_K(E)$ which is finitely generated (i.e., there exists $x_1, x_2, \dots, x_n \in R$ for which $I = \sum_{i=1}^n R x_i R$).  Then $I$ is principal (i.e., there exists $x\in R$ with $I = RxR$.)
\end{corollary}
\begin{proof}   Consider the (finitely generated) left ideal $T = \sum_{i=1}^n R x_i $ of $R$.  Then $I = TR$ as two-sided ideals.    By Corollary \ref{LpasBezoutallgraphs} there exists $x\in R$ with $T = Rx$.   Then $I = TR = RxR$, as desired. 
\end{proof}

\begin{remark}\label{Moritaequivremark}
{\rm In \cite{AR} it is shown that if $v$ is a non-isolated source vertex, then $L_K(E)$ is Morita equivalent to $L_K(E\setminus \{v\})$.   Although at first glance one might think that this information could be  useful in  establishing Case 2 in the proof of Theorem \ref{LpasareBezout},  it  indeed is not helpful in this context,  because the \Bez \ property is in general not a Morita invariant.  (See e.g. \cite[page 536]{W} and/or \cite[page 625]{Rob}.    We thank L. Small for pointing out these references.)  \hfill $\Box$ 
}

\end{remark}

We close the article with a remark about the usefulness of the main result.   
{\rm In \cite{AMT} we use the now-established \Bez \ property of Leavitt path algebras to investigate the injectivity of a naturally-occuring class of modules over $L_K(E)$,  by means of a divisibility notion which arises in this setting.    The key tool is the fact that, in a \Bez \ ring, the Baer test on finitely generated ideals reduces to a divisibility property.}


%
%

\medskip


\end{document}